\documentclass{amsart}

\usepackage{graphicx} % Required for inserting images
\usepackage{macro}
\usepackage{tikz,tkz-euclide}
\usepackage{mathtools}
\usepackage{thmtools}
\usepackage{thm-restate}

\title{The Hochschild Homology of Reedy Categories}
\author{Alex Ballow \\ Montana State University}

\date{April 2024}

\begin{document}

\begin{abstract}
    We calculate the Hochschild homology of generalized Reedy categories, such as the simplex category, the category of finite sets, the category of finite-dimensional categories, and the PROP associated to an operad.
\end{abstract}

\maketitle

\tableofcontents

\section*{Introduction}
In this paper we aim to calculate a version of Hochschild homology of generalized Reedy categories. Hochschild homology traditionally is evaluated on associative algebras over rings. However Hochschild homology can be calculated using just the bar construction and tensor product, both of which have an analog in general monoidal categories. This process, described more explicitly in \cref{Hhomology}, gives a generalization of Hochschild homology which can be calculated for any arbitrary category. In this generality calculating Hochschild homology is not a tractable problem, so we focus on calculating it for generalized Reedy categories. Generalized Reedy categories are filtered categories where every morphism has a factorization. 

 The category of finite sets, $\Fin$, is a central example of generalized Reedy categories. The filtration is given by cardinality and morphism factor around the image. This example is especially motivating for our calculation as Hochschild homology approximates K-theory and the K-theory of $\Fin$ is known to be the sphere spectrum \cite{barratt1972homology}. We explicitly calculate the Hochschild homology for this and a few other generalized Reedy categories in corollaries for ease of reference. The Reedy structure for most of our chosen examples is well known, but we also include a less common example of a Reedy category: the PROP associated to an operad in {\sf Sets}. The specific examples will be calculated using the more general result from the following theorem. 

\begin{restatable*}{theorem}{test}
\label{thm}
    Given a generalized Reedy category, $\cR$, there is an equivalence between spaces: 
    \[
    {\sf HH}(\cR)  \simeq \coprod_{n \ge 0} \coprod_{[r]\in\cR_{= n}} \coprod_{[\phi]\in \Aut_{\cR}(r)_{/^{\sf conj}}} \sB \sZ_{\Aut_{\cR}(r)}(\phi)
    \]
    where $\sZ_{\Aut_{\cR}(r)}(\phi)$ is the centralizer of $\phi$ in $\Aut_{\cR}(r)$ and $[r]\in\cR_{= n}$ denotes a choice of representative $r$ for each isomorphism class of objects in $\cR_{= n}$.
    \label{mainthm}
\end{restatable*}

As an immediate corollary, we calculate the Hochschild homology for the common Reedy categories listed in \cref{Reedycats}.
\begin{corollary}
An application of \cref{mainthm}
\begin{eqnarray*} 
    &\hh(\Fin) \simeq \coprod\limits_{n\geq 0} \coprod\limits_{\substack{ n \hspace{.5ex} \text{cycle} \\ \sf \text{type} \hspace{.5ex} \omega}} \sB \sZ_{\Sigma_n}(\omega)~;  \\
    &\hh(\finvect) \simeq \coprod\limits_{n\geq 0}  \coprod\limits_{[M]\in GL_{n{/^{\sf conj}}}}\sB \sZ_{GL_n}(M)~; \\
    &\hh(\bdelta) \simeq \ZZ_{\ge 0}~;  \\
    &\hh(\para) \simeq \ZZ_{\ge 0} \times \ZZ~; \\
    &\hh(\Lambda) \simeq \coprod\limits_{n\geq 0} \ZZ_{/n\ZZ}~; \\
    &\hh(PROP({\sf Assoc})) \simeq \ZZ_{\ge 0}. 
\end{eqnarray*}
where {\sf Assoc} is the associative operad.
\end{corollary}

The layout of the remaining content begins with an overview of necessary background material in Section 1. Specifically, it gives a review of generalized Reedy categories before discussing formulism established in  \cite{ayala2025symmetriescyclicnerve}. The formulism develops quiver representations of $\infty$-categories and leverages the inherently cyclic nature of Hochschild homology. In Section 2 we prove the theorem through a series of inductive arguments. The proof begins with induction on the degree of the Reedy structure. Within that inductive step we induct again, this time over the filtration of $\para$. The second inductive argument constructs a map with disjoint image such that the first inductive hypothesis applies to a piece. The map is homotopic to the identity, so the image is homotopy equivalent to the domain. This proves the inductive step of the first induction, finishing the proof.
\subsection{Acknowledgments}
This material is based upon work supported by the U.S. Department of Energy, Office of Science, Office of Advanced Scientific Computing Research, Department of Energy Computational Science Graduate Fellowship under Award Number DE-SC0022158. Additionally this work would never have been completed without the help of my advisor at Montana State University, David Ayala, who was supported by the National Science Foundation under awards 1812055 and 1945639.

\section{Definitions, Notation, and Preliminary Results}
\subsection{Reedy Categories}
 We begin with some facts about generalized Reedy categories. 
\begin{definition}
     A {\bf generalized Reedy category}, or just {\bf Reedy category}, is a category $\cR$ together with two subcategories $\cR^-$ and $\cR^+$ and a function $\deg:\Obj(\cR) \to \ZZ_{\ge 0}$ called degree satisfying the following:
     \begin{enumerate}
     \label{Reedydef}
         \item given a morphism $f:X \to Y$ in $\cR^+$, then $\deg(X) \le \deg(Y)$ i.e. it raises degree;
         \item given a morphism $f:X \to Y$ in $\cR^-$,  then $\deg(X) \ge \deg(Y)$ i.e. it lowers degree;
         \item given a morphism $f:X \to Y$, then $f$ is an isomorphism if and only if $f\in \cR^+\cap \cR^-$ i.e. isomorphisms preserve degree.\label{3}
         \item every morphism $f $ in $ \cR$ factors uniquely up to unique isomorphism as a morphism $f^- $ in $ \cR^-$ followed by a morphism $f^+ \in \cR^+$.
         \label{five}
     \end{enumerate}
     The full subcategory consisting of those objects with degree less than (greater than) or equal to $n \in \ZZ_{\ge 0}$ will be denoted $\cR_{\le n}$ ($\cR_{\ge n}$). The subcategory generated by objects with degree $n$ and degree preserving morphisms will be denoted $\cR_{= n}$. Finally the intermediary object in a factorization of $f$ will be denoted ${\sf im}(f) $. \cref{five} above then implies that for every $f:X \to Y$ in $ \cR$ we have the following communitive diagram:
     \[\begin{tikzcd}
        	X && Y \\
        	& {\im(f)}
        	\arrow["f", from=1-1, to=1-3]
        	\arrow["{f^-}"', from=1-1, to=2-2]
        	\arrow["{f^+}"', from=2-2, to=1-3]
        \end{tikzcd}~.\]
\end{definition}
\begin{remark}
    There is some variance in what exactly the definition of a Reedy category should be. In deciding which definition to use we prioritize future generalization. Namely, we nearly use the generalized Reedy category definition used in \cite{Berger_2010}(Definition 1.1). We drop their condition (iv), which required certain isomorphisms to be the identity. This justifies our use of $\ZZ_{\ge 0}$ rather than an arbitrary ordinal. We drop this condition to enable \cref{Reedyinf}.   
\end{remark}
\begin{remark}
    The above definition of a Reedy category manifestly affords an extension to $\infty$-categories. We believe the results presented here are valid in that generalization. 
    \label{Reedyinf}
\end{remark}
\begin{observation}
    It follows from \cref{Reedydef} (\ref{3}) that $\cR^+$ and $\cR^-$ each contain all of the objects of $\cR$, i.e. they are wide subcategories.
\end{observation}
    Before we proceed let us introduce a few examples of Reedy categories. We will explicitly calculate the factorization homology of all of the following.
\begin{example}
The following are all Reedy categories:
\label{Reedycats}
    \begin{enumerate}
        \item $\Fin$ - The category of finite sets where the degree is given by cardinality, $\Fin^+$ consists of surjections, $\Fin^-$ consists of injections and the factorization required by \cref{five} is the factorization of set maps through their image;
        \item $\finvect$ - The category of finite vector spaces over a field, with degree given by dimension and the factorization of a morphism $f$ is again given by a surjection into the image of $f$ and an injection into the codomain;
        \item $\bDelta$ - The simplex category where degree is given by $[p]\mapsto p$ and the factorization is again surjections and injections through the image;
        \item $\para$ - The paracyclic category (\cref{paracyclic}) where the degree is given by the cardinality of $[i, 1 \cdot i]$ and the factorization is again through the image;
        \item $\Lambda$ - Connes cyclic category as defined in \cite{connes1983cohomologie} where the degree is given by the cardinality of the cyclically ordered finite set and the factorization is through the image;
        \item The PROP associated to an operad in {\sf Sets}.
        \label{PROP}
    \end{enumerate}
\end{example}
The Reedy structure was named for all the examples except \cref{Reedycats}(\ref{PROP}), which we will discuss more thoroughly. Intuitively, an operad is a method to describe algebraic structure, or $n$-ary operations, in a monoidal category. PROPs generalize this structure to include the possibility of operations with multiple outputs.
% \begin{definition}
%     Given a symmetric monoidal category, ($\cV$, $\otimes$, $I$), an operad, $\cO$, in $\cV$ is objects $\cO(n)$ from $\cV$ indexed by $\NN$ together with
%     \begin{itemize}
%         \item Right actions of symmetric groups $\rho_n:S_n \to \hom(\cO(n),\cO(n))$;
%         \item A unit $e:I \to \cO(1)$ ;
%         \item Composition operations
%         $\cO(k)\otimes \cO({n_1})\otimes \cO(n_2) \otimes ... \otimes \cO({n_k})\to \cO({n_1+…+n_k})$.
%     \end{itemize}
%     The algebraic structure comes when considering the $\cO(n)$ as objects which parameterize $n$-ary operations.
% \end{definition}
% This definition lends itself well to generalization, we focus on allowing for morphisms more general than $n$-ary operations. Specifically we allow for not only for multiple inputs, but also multiple outputs. This concept is called a PROP. 

\begin{observation}
    Given an operad $\cO$ in $\Sets$, $PROP(\cO)$ is a generalized Reedy category with $PROP(\cO)^-$ ($PROP(\cO)^+$) the subcategory whose morphisms are carried to decreasing (increasing) morphisms in $\Fin$. These subcategories do form a factorization system since $\cO$ is unital.
\end{observation}
\begin{remark}
    The morphisms in $PROP(\cO)$ can be depicted as trees, with the inputs as the leftmost nodes and the outputs as the rightmost nodes. Composition in this depiction corresponds to concatenation. For example  here is the tree representing a morphism $5 \to 3$, specifically of $(f_1, f_2, f_3) \in \cO(3) \times \cO(0) \times \cO(2)$:
   
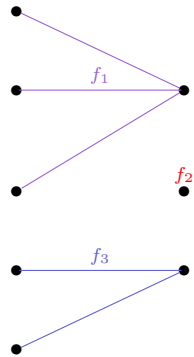
\[\begin{tikzcd}
	\bullet \\
	\bullet && \bullet \\
	\bullet && {\stackrel{{\color{red}{f_2}}}{\bullet}} \\
	\bullet && \bullet \\
	\bullet
	\arrow[shorten <=-6pt, shorten >=-6pt, color={rgb,255:red,153;green,92;blue,214}, no head, from=1-1, to=2-3]
	\arrow[shorten <=-6pt, shorten >=-6pt, "{f_1}", color={rgb,255:red,153;green,92;blue,214}, no head, from=2-1, to=2-3]
	\arrow[shorten <=-6pt, shorten >=-6pt, color={rgb,255:red,153;green,92;blue,214}, no head, from=3-1, to=2-3]
	\arrow[shorten <=-6pt, shorten >=-6pt, "{f_3}", color={rgb,255:red,92;green,92;blue,214}, no head, from=4-1, to=4-3]
	\arrow[shorten <=-6pt, shorten >=-6pt, color={rgb,255:red,92;green,92;blue,214}, no head, from=5-1, to=4-3]
\end{tikzcd}\]
This makes the Reedy category structure obvious. The factorization of a morphism $f$ is given by first doing any non-nullary operations followed by the identity and any necessary nullary operations. As an example we factor the morphism $(f_1, f_2, f_3) \in \cO(3) \times \cO(0) \times \cO(2)$ from above. 
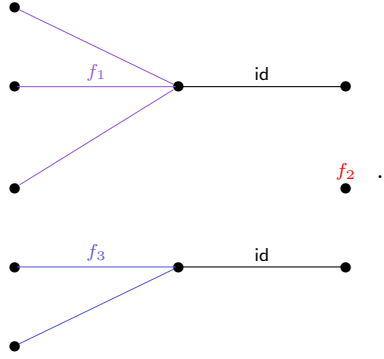
\[\begin{tikzcd}
	\bullet \\
	\bullet && \bullet && \bullet \\
	\bullet &&&& {\stackrel{{\color{red}{f_2}}}{\bullet}} \\
	\bullet && \bullet && \bullet \\
	\bullet
	\arrow[shorten <=-6pt, shorten >=-6pt, color={rgb,255:red,153;green,92;blue,214}, no head, from=1-1, to=2-3]
	\arrow[shorten <=-6pt, shorten >=-6pt, "{f_1}", color={rgb,255:red,153;green,92;blue,214}, no head, from=2-1, to=2-3]
	\arrow[shorten <=-6pt, shorten >=-6pt, "{\sf id}", no head, from=2-3, to=2-5]
	\arrow[shorten <=-6pt, shorten >=-6pt, color={rgb,255:red,153;green,92;blue,214}, no head, from=3-1, to=2-3]
	\arrow[shorten <=-6pt, shorten >=-6pt, "{f_3}", color={rgb,255:red,92;green,92;blue,214}, no head, from=4-1, to=4-3]
	\arrow[shorten <=-6pt, shorten >=-6pt, "{\sf id}", no head, from=4-3, to=4-5]
	\arrow[shorten <=-6pt, shorten >=-6pt, color={rgb,255:red,92;green,92;blue,214}, no head, from=5-1, to=4-3]
\end{tikzcd}~.\]
Explicitly, this picture is denoting the composition (as in the definition of operad) of $(f_1, f_3) \in \cO(3) \times \cO(2)$ with $(\id, f_2, \id) \in \cO(1) \times \cO(0) \times \cO(1)$. 
Note that this analysis relies on $\cO(1)$ being a discrete category which is true of operads over sets. We believe the results in this paper apply to PROPs of operads over more general categories (i.e. $\spaces$), but we omit any such results here.  
\end{remark}

Now we get back on track with some useful properties of Reedy categories. 
\begin{observation}
    The opposite of a Reedy category, $\cR$ is also a Reedy category. Namely the degree map is unchanged, $(\cR^{\sf op})^+=\cR^-$, and $(\cR^{\sf op})^-=\cR^+$. 
\end{observation}
\begin{lemma}
    Given morphisms $A \xra{f} B\xra{g} C$ in a Reedy category, $\cR$, such that $\deg(B) < n$, then $\deg(\im(f)), \deg(\im(g)) < n$.
    \label{factor}
\end{lemma}
\begin{proof}
    We begin by factoring both $f$ and $g$:
    \[
        A \xra{f^-} \im(f) \xra{f^+} B\xra{g^-} \im(g) \xra{g^+} C ~.
    \]
    Now as $\deg(B) < n$ and $ \im(f) \xra{f^+} B$ is an increasing morphism then $\im(f)$ must have had degree less than $n$. Similarly, as $B \xra{g^-} \im(g)$ is decreasing, $\deg(\im(g)) < n$.
\end{proof}

\begin{lemma}
    The factorization in \cref{Reedydef} (\ref{five}) is functorial. In other words, given factorizations of $f:X \to X'$ and  $g:Y \to Y'$ along with functors $\eta:X \to Y$ and $\eta':X' \to Y'$ such that their square commutes then the dotted arrow exists and makes the following diagram commute
       
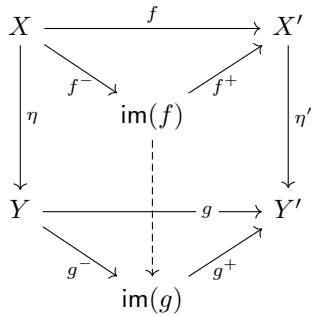
\[\begin{tikzcd}
    	X && {X'} \\
    	& {{\sf im} (f)} \\
    	Y && {Y'} \\
    	& {{\sf im} (g)}
    	\arrow["f", from=1-1, to=1-3]
    	\arrow["f^-" {below}, from=1-1, to=2-2]
    	\arrow["\eta", from=1-1, to=3-1]
    	\arrow["{\eta'}", from=1-3, to=3-3]
    	\arrow["f^+" {below}, from=2-2, to=1-3]
    	\arrow[dashed, from=2-2, to=4-2]
    	\arrow["{g}" {description, near end}, from=3-1, to=3-3]
    	\arrow["g^-" {below}, from=3-1, to=4-2]
    	\arrow["g^+" {below}, from=4-2, to=3-3]
        \end{tikzcd}\hspace{5ex}.\]
    \label{imcommute}
    \end{lemma}
\begin{proof}
    Consider the subsquare
    \[\begin{tikzcd}
	X && {\im(g)} \\
	\\
	{\im(f)} && {Y'} 
	\arrow["{g^-\circ \eta}", from=1-1, to=1-3]
	\arrow["{f^-}"', from=1-1, to=3-1]
	\arrow["{g^+}", from=1-3, to=3-3]
	\arrow["{\eta'\circ f^+}"', from=3-1, to=3-3]
\end{tikzcd}\]
and factor the horizontal morphisms
\[\begin{tikzcd}
	X & {\im(g^-\circ \eta)} & \\
	{\im(f)} & {\im(g)} \\
	{\im(\eta'\circ f^+)} & {Y'}
	\arrow["{(g^-\circ \eta)^-}", from=1-1, to=1-2]
	\arrow["{f^-}"', from=1-1, to=2-1]
	\arrow["{(g^-\circ \eta)^+}", from=1-2, to=2-2]
	\arrow["{(\eta'\circ f^+)^-}"', from=2-1, to=3-1]
	\arrow["{g^+}", from=2-2, to=3-2]
	\arrow["{(\eta'\circ f^+)^+}"', from=3-1, to=3-2]
\end{tikzcd}~.\]
Since the factorization system is unique up to isomorphism, there exists an isomorphism $$\phi:\im(\eta'\circ f^+) \xra{\simeq} \im(g^-\circ \eta)$$ making the diagram commute. Then there is a morphism $$(g^-\circ \eta)^+\circ \phi \circ (\eta'\circ f^+)^-:\im(f) \to \im(g) $$ as desired.
\end{proof}
\begin{lemma}
    Let $\mathbb{Z}_{\ge 0} \to \Cat $ such that $N \mapsto \cC_N$ be a functor. Then the canonical maps $\Obj(\underset{N \to \infty}\colim \cC_N) \leftarrow \underset{N \to \infty}\colim\Obj(\cC_N)$ and $\mathsf{Mor} (\underset{N \to \infty}\colim \cC_N) \leftarrow \underset{N \to \infty}\colim\mathsf{Mor}(\cC_N)$ are equivalences.
    \label{valuewisecolim} 
\end{lemma}
\begin{proof}
Recall the fully faithful functor $\Cat \xhookrightarrow{N} \Psh(\bDelta)$ from Rezk \cite{rezk2001model} with image consisting of the complete Segal spaces. The Segal condition and the completeness conditions on a simplicial space each stipulate preservation of finite limit diagrams. Therefore by the definition of filtered categories, a filtered colimit of complete Segal simplicial spaces is again a complete Segal simplicial space. In other words, the fully faithful functor $\Cat \xhookrightarrow{N} \Psh(\bDelta)$ preserves filtered colimits. Since colimits in a presheaf category are computed point-wise, the desired result follows from the identities $N\cC[0]=\Obj(\cC)$ and $N\cC[1]=\Mor(\cC)$.
\end{proof}
This is immediately applicable to Reedy categories.
\begin{observation}
    Let $\cR $ be a Reedy category. The canonical functor $\underset{N \to \infty}{\colim}\cR_{\le N}\to \cR$ is an equivalence. Indeed, \cref{valuewisecolim} applied to objects implies this functor is essentially surjective and \cref{valuewisecolim} applied to morphisms gives full and faithfulness.
    \label{Reedyvaluewisecolim}
\end{observation}

\subsection{Hochschild Homology}
\label{Hhomology}
First recall that the conventional definition of Hochschild homology is the homology of the Hochschild chain complex defined below.
\begin{definition}
Given a ring $A$, its {\bf Hochschild chain complex, ${\sf HC_{\bullet}(A)}$}, is defined by ${\sf HC}_n(A):=A^{\otimes n+1}$ and the differential $b_n := \Sigma^n_{j=0}(-1)^jd_j$. Here the maps $d_i:{\sf HC}_n(A)\to {\sf HC}_{n-1}(A)$ are as follows:
\[
d_i(a_0\otimes a_1\otimes ... \otimes a_n)=
\begin{cases} 
      a_0a_1\otimes a_2\otimes ... \otimes a_n & i=0 \\
      a_na_0\otimes a_1\otimes ... \otimes a_{n-1} & i=n \\
      a_0\otimes a_1\otimes ... \otimes a_i a_{i+1}\otimes ...\otimes a_n & \text{otherwise}
   \end{cases}
\]
for each $n \in \ZZ_{\ge 0}$. 
\end{definition}

In order to generalize from rings to categories, it helps to look at Hochschild homology through the lens provided by the Dold-Kan correspondence \cite{b1ec9719-9e94-309b-b5d6-db2ea7da72f4}. Recall that the Dold-Kan correspondance is an equivalence between the category of (nonnegatively graded) chain complexes and the category of simplicial abelian groups. Through this viewpoint, the chain complex ${\sf HC}_{\bullet}(A)$ corresponds to a simplicial abelian group: the {\bf cyclic bar construction}. The cyclic bar construction can then be generalized to $\infty$-categories. This perspective motivates \cite{ayala2025symmetriescyclicnerve} Definition 3.1.2 of the Hochschild homology of a $\infty$-category, which is a colimit of a simplicial object.

We follow the framework developed in \cite{ayala2025symmetriescyclicnerve} where they understand $\infty$-categories by probing them with quivers. 
\begin{definition}
    A {\bf quiver} is a category which is the free category on a finite directed graph. The category $\quiv$ is the full subcategory $\quiv \subset \Cat$ consisting of those categories that are quivers.
    \label{factorizationhomology}
\end{definition}

The cyclic nature of Hochschild homology is well-documented, and can be used as a computational tool. This cyclicity can be realized through Definition 2.4.1 in \cite{ayala2025symmetriescyclicnerve}. 
\begin{definition}
    An object in the {\bf paracyclic category $\para$} is a nonempty linearly ordered set $I$ for which, for each $i < j$ in $I$ the interval $[i,j] \subset I$ is finite and equipped with an action by the additive group $\ZZ$ for which for each $i \in I$ there is a relation $i < 1 \cdot i$ for all $i \in I$. A morphism in $\para$ is a $\ZZ$-equivariant (weakly) order-preserving map. Composition is composition of maps. Identities are identity maps. 
    \label{paracyclic}
\end{definition}
This category admits a natural functor as described in Notation 3.1.1 in \cite{ayala2025symmetriescyclicnerve}
\begin{align*}
    \para &\to \quiv \stepcounter{equation}\tag{\theequation} \label{functor}\\
    \lambda = (\ZZ \curvearrowright I) &\mapsto \overline \lambda := I_{/ \ZZ}
\end{align*}
where $I_{/ \ZZ}$ is the cyclically directed graph of the $\ZZ$ quotient. For example if the interval $[i, j]$ is given by $\{i=:i_1, ..., i_k:=1 \cdot i\}$ then $I_{/ \ZZ}$ looks like
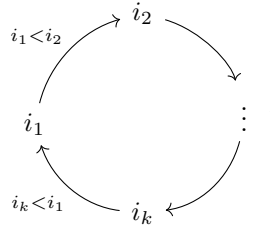
\[\begin{tikzcd}
	& {i_2} \\
	{i_1} && \vdots \\
	& {i_k}
	\arrow[curve={height=-12pt}, from=1-2, to=2-3]
	\arrow["{i_1<i_2}", curve={height=-12pt}, from=2-1, to=1-2]
	\arrow[curve={height=-12pt}, from=2-3, to=3-2]
	\arrow["{i_k<i_1}", curve={height=-12pt}, from=3-2, to=2-1]
\end{tikzcd}~~~~.\]
Additionally we have the functor 
\begin{align*}
    \bdelta &\to \para \\
    [p]&\mapsto \chi_p := [p]^{\star \ZZ} ~.
    \stepcounter{equation}\tag{\theequation} \label{Chip}
\end{align*}
where $\star$ is the join of posets, so $[p]^{\star \ZZ}$ is the $\ZZ$-fold join of $[p]$.
Composing, \cref{functor} and \cref{Chip}, gives rise to a new functor, 
\[
\chi:\bDelta \to \para \to \quiv~.
\]
\begin{definition}
(\cite{ayala2025symmetriescyclicnerve} Definition 4.1.1) The {\bf Hochschild homology of a category $\cC$, $\hh(\cC)$} is the space
    \[
    \hh(\cC) :=
\colim \left( \bDelta^{\sf op} \xrightarrow{\chi} {\sf Quiv}^{\sf op} \xrightarrow{\hom_\Cat(\_, \cC)} \Gpd \xrightarrow{|\_|} \spaces \right)
    \]
    where $|\_|$ is the geometric realization functor.
    \label{defhh}
\end{definition}

 Towards working with this definition, recall the unstraightening construction functor $\Un: {\sf Fun}(\cC,\Spaces) \longrightarrow \Cat_{(\infty,1)/\cC}$ and \cite{lurie2008highertopostheory} Corollary 3.3.4.6.  
\begin{lemma}
    Let $\cC$ be a category and $F \in {\sf Fun}(\cC,\Spaces)$ be a functor into $\spaces$. Then there is a canonical equivalence between spaces $\colim(F) \simeq |\Un F|$~.
    \label{colimandun}
\end{lemma}
We will not discuss unstraightening in full generality.  We will quickly review the simpler Grothendieck construction. The unstraightening is an $\infty$-categorical version of the Grothendieck construction for functors from an ordinary category to $\Gpd$. For a more comprehensive treatment of unstraightening see chapter 3 section 2 in \cite{lurie2008highertopostheory}.

\begin{definition}
    Let $F:\cD\to \Gpd$ be a functor from any category to the groupoids of categories. The {\bf Grothendieck construction} for $F$  is the category $\Un(F)$ in which an object is a pair $(d,x)$, with $d \in \Obj(\cD)$  and $x \in \Obj(F(d))$; and in which a morphism from $(d_1,x_1)$ to $(d_2,x_2))$ is a pair $(f,g)$ where $f:d_1 \to d_2$ in $\cD$, and $g:F(f)(x_1)\to x_2$ in $F(d_2)$. Composition of morphisms is defined as $(f,g)\circ(f',g')=(f\circ f',g\circ F(f)(g'))$.
\end{definition}
The domain category plays a crucial role in this construction, so we look to simplify it through the use of final functors, starting with \cite{ayala2025symmetriescyclicnerve} Lemma 2.4.5 and Lemma 3.6.5.
 \begin{lemma}
     $\bDelta^{\sf op} \xrightarrow{(\ref{Chip})} \para^{\sf op}$ is a final functor. Therefore the canonical map between spaces 
     \label{finalfun}

\[
    \hh(\cC) := \colim \left( \bDelta^{\sf op} \xrightarrow{\chi} {\sf Quiv}^{\sf op} \xrightarrow{\hom_\Cat(\_, \cC)} \spaces \right) 
    \xrightarrow{\simeq} \colim \left( \para^{\sf op} \to {\sf Quiv}^{\sf op} \xrightarrow{\hom_\Cat(\_, \cC)} \spaces \right) 
    \label{defhhpara}
\]
is an equivalence.
 \end{lemma} 
\begin{lemma}
  The map $\para \xrightarrow{(\ref{functor})} \quiv_{/S^1}$ is final. 
\end{lemma}
\begin{remark}
  These two lemmas imply the canonical map between spaces,  
  \[
    \colim \left( \bDelta^{\sf op} \xrightarrow{\chi} {\sf Quiv}^{\sf op} \xrightarrow{\hom_\Cat(\_, \cC)} \spaces \right) 
    \xrightarrow{\simeq} \colim \left( \quiv_{/S^1}^{\sf op} \to {\sf Quiv}^{\sf op} \xrightarrow{\hom_\Cat(\_, \cC)} \spaces \right) =: \int_{S^1} \cC ~,
\]
is an equivalence.
Here, $\int_{S^1} \cC$ is the factorization homology of $\cC$ over $S^1$. This is Definition 4.1.1 in \cite{ayala2025symmetriescyclicnerve} with $M=S^1$. The results of this paper are also relevant to those looking to calculate the factorization homology over the circle of Reedy categories due to the following corollary.
\end{remark}
\begin{corollary} 
Let $\cC$ be a category. Then there is a canonical equivalence between spaces
    \[
    \hh(\cC) \simeq 
    \int_{S^1}\cC ~.
\]
\label{HH=fact}
\end{corollary}
% This is a simplification as $\para^{\sf op}$ is easier to compute with as objects can be explicitly described. Before looking more closely at $\para$, we can in fact take this one step further with the final functor $\bDelta^{\sf op} \to \para^{\sf op}$. This final functors along with the one discussed above allow us to relate factorization homology with the colimit of a simplicial abelian group which by Dold-Kan is equivalent to a chain complex. If you follow through these equivalences you will see that the resulting chain complex is related to the Hochschild chain complex motivating the following definition. 
 These lemmas also allow us to replace our domain category with $\para$. The objects of $\para$ have a nice representation, as seen in Lemma 2.4.3 \cite{ayala2025symmetriescyclicnerve}.
\begin{lemma}
    Every object in $\para$ is non-canonically isomorphic with $\frac{1}{\ell} \ZZ$ for some $\ell \in \NN$.
    \label{paraobjects}
\end{lemma}

The unstraightening of our functor now has a more explicit form. 
\begin{observation}
        \label{Un<n}
        For $\cR$ a Reedy category,
         {\bf $\Un \left( \para^{\sf op} \to {\sf Quiv}^{\sf op} \xrightarrow{\hom_\Cat(\_, \cR_{\le n})} \spaces \right)$} is an ordinary category where an object is a pair $(\lambda, \ell)$ with $\lambda \in \para$ and $\ell:\overline{\lambda} \to \cR_{\le n}$ a functor. A morphism $(\lambda, \ell) \to (\lambda', \ell')$ then is a morphism $\lambda \xra{f} \lambda'$ in $\para$ and a natural isomorphism: 
    \[\begin{tikzcd}
	{\overline\lambda} &[-1em]&[-1em] {\overline{\lambda'}} \\[-1em]
	& \rotatebox[origin=c]{30}{$\overset{\simeq}{\implies}$}  \\
	& \cR_{\le n}
	\arrow["{\overline f}", from=1-1, to=1-3]
	\arrow["\ell"', from=1-1, to=3-2]
	\arrow["{\ell'}", from=1-3, to=3-2]
    \end{tikzcd}.\]
    \end{observation}
    \begin{definition}
    For the sake of brevity for $n \in \NN$ we let
    \[
    \Un_{\le n} := \Un \left( \para^{\sf op} \to {\sf Quiv}^{\sf op} \xrightarrow{\hom_\Cat(\_, \cR_{\le n})} \spaces \right)~.
    \]

    \label{defun<n}
    \end{definition}
    Objects in $\Un_{\le n}$ may be depicted as follows.
    \begin{remark}
        We can depict an object, $(\lambda := \frac{1}{3} \ZZ, \ell :\overline{\frac{1}{3} \ZZ} \to \cR_{\le n})$, in $\Un_{\le n}$ pictorially by first noting $\overline{\frac{1}{3} \ZZ}$ is the free quiver generated by the digraph
        
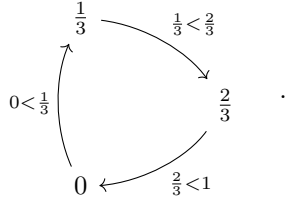
\[\begin{tikzcd}
    	{\frac{1}{3}} \\[-1ex]
    	&[-1ex]&[-1ex] {\frac{2}{3}} \\[-1ex]
    	0
    	\arrow[curve={height=-10pt},"{\frac{1}{3}<\frac{2}{3}}", from=1-1, to=2-3]
    	\arrow[curve={height=-10pt},"{\frac{2}{3}<1}", from=2-3, to=3-1]
    	\arrow[curve={height=-10pt},"{0<\frac{1}{3}}", from=3-1, to=1-1]
        \end{tikzcd} \hspace{3ex}.\]
   Then if $\ell$ takes $0,\frac{1}{3}, $and$ \frac{2}{3}$ to $X, Y, Z \in \Obj(\cR_{\le n})$ and $0<\frac{1}{3}$, $\frac{1}{3}<\frac{2}{3}$, and $\frac{2}{3}<1$ to $f,g,h \in {\sf Mor}(\cR_{\le n})$ $\ell$ can then be depicted as
        
         
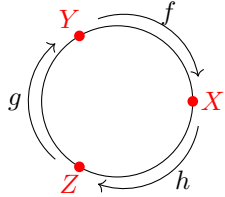
\[
    \begin{tikzpicture}

    % define a circle (center=O and \radius)
      \coordinate (T) at (0,0);
      \def\radius{1cm}
    
      % draw this circle and its center
      \draw (T) circle[radius=\radius];
    
      % define a random point (A) on this circle
      \path (T) ++(0:\radius) coordinate (0);
      \path (T) ++(120:\radius) coordinate (120);
      \path (T) ++(240:\radius) coordinate (240);
      \path (T) ++(60:\radius) coordinate (60);
      \path (T) ++(180:\radius) coordinate (180);
      \path (T) ++(310:\radius) coordinate (310);
      \path (T) ++(10:\radius) coordinate (10);
      \path (T) ++(130:\radius) coordinate (130);
      \path (T) ++(250:\radius) coordinate (250);
    
      \fill[red] (0) circle[radius=2pt] ++(0:.75em) node {$X$};
      \fill[red] (120) circle[radius=2pt] ++(120:.75em) node {$Y$};
      \fill[red] (240) circle[radius=2pt] ++(240:.75em) node {$Z$};

      \fill[black] (60) ++(60:1em) node {$f$};
      \fill[black] (180) ++(180:1em) node {$g$};
      \fill[black] (310) ++(310:1em) node {$h$};

      \draw[<-] (10) ++(60:.5em) arc[radius=\radius, start angle=10, end angle=110];
      \draw[<-] (130) ++(180:.5em) arc[radius=\radius, start angle=130, end angle=230];
      \draw[<-] (250) ++(300:.5em) arc[radius=\radius, start angle=250, end angle=350]; 
    \end{tikzpicture}
    \label{picture}
    .\]
    The number and position of the points on the circle are determined by the element in $\para$. Since $\ell$ determined the labels of the picture we call it a labeling functor. 
    \label{Unaspic}
    \end{remark}
    By evoking definitions and \cref{colimandun} we see the following.
    \begin{observation}
        \label{Unn=HH}
     There is a canonical equivalence between spaces:
    \begin{align*}
        {\sf HH}(\cR_{\le n})   
        &\hspace{1ex} \stackrel{(\ref{defhhpara})}{\simeq} \colim \left( \para^{\sf op} \to {\sf Quiv}^{\sf op} \xrightarrow{\hom_\Cat(\_, \cR_{\le n})} \spaces \right)
        \stepcounter{equation}\tag{\theequation} \label{line4}\\
          &\stackrel{(\ref{colimNUn})}{\simeq}  \left|\Un \left( \para^{\sf op} \to {\sf Quiv}^{\sf op} \xrightarrow{\hom_\Cat(\_, \cR_{\le n})} \spaces \right)\right| \\ 
        &\hspace{2.5ex}  \overset{Def \ref{defun<n}}{\simeq} |\Un_{\le n}|~.
    \end{align*}
    \end{observation}
\subsection{Loopspace of a Groupoid}
We now note a few additional facts, starting with the following theorem.
\begin{prop}
    Let $G$ be a group. There is a canonical equivalence:
    \[
    \map(S^1, \sB G) \simeq \coprod_{[g]\in G_{/\sf conj}}\sB\sZ_G(g)
    \]
    \label{MapS1}
\end{prop}
 Setting $G$ set to be $\Aut_{\cR}(r)$ gives the following corollary since the stabilizers are given by the centralizers of automorphisms $\phi$, denoted $\sZ_{\Aut_{\cR}(r)}(\phi)$. 
\begin{corollary}
    For a Reedy category $\cR$
\[ 
\hom(S^1,\sB \Aut_{\cR}(r)) \simeq \coprod_{[\phi] \in \Aut_{\cR}(r)_{/{\sf conj}}}\sB \sZ_{\Aut_{\cR}(r)}(\phi)~.
\] 
\label{homcent}
\end{corollary}
Before proving this theorem we introduce some more definitions and lemmas.
\begin{definition}
    Let $G \curvearrowright X$ be a group acting on a set. The {\bf action groupoid, $X_{/\! / G}$}, is the groupoid where an object is an element in $X$, and a morphism from $x$ to $x'$ is an element $g\in G$ such that $g \cdot x = x'$.
\end{definition}
\begin{lemma}
    Let $G$ and $H$ be groups. There is the canonical equivalence between groupoids:
    \[
    \hom(\sB H, \sB G) \xleftarrow{\simeq} \hqconj{\hom(H, G)}{G}~.
    \]
    \label{homHG}
\end{lemma}
\begin{proof}
    Consider the functor 
    \[
    \hqconj{\hom(H, G)}{G} \to \hom(\sB H, \sB G)
    \]
    given by 
    \[
    f \mapsto \begin{cases}
        * \mapsto * \\
        h \mapsto f(h)
    \end{cases}
    \]
    on objects and on morphisms
    \[
    g \mapsto \eta_g:
    \]
    where $\eta_g$ is the natural transformation 
    \[\begin{tikzcd}
	{\sB H} && {\sB G}
	\arrow[""{name=0, anchor=center, inner sep=0}, "f"', curve={height=12pt}, from=1-1, to=1-3]
	\arrow[""{name=1, anchor=center, inner sep=0}, "h", curve={height=-12pt}, from=1-1, to=1-3]
	\arrow["\eta_g"', Rightarrow, from=0, to=1]
\end{tikzcd}\]
    with $\eta_g(*)=g$. This functor is a categorical equivalence. To see this consider the condition for $\eta_g$ to be a natural transformation: for all $a \in \hom(H, G)$
    \[\begin{tikzcd}
	{f(*)} && {h(*)} \\
	& \circlearrowright \\
	{f(*)} && {h(*)}
	\arrow["g", from=1-1, to=1-3]
	\arrow["{f(a)}"', from=1-1, to=3-1]
	\arrow["{h(a)}", from=1-3, to=3-3]
	\arrow["g"', from=3-1, to=3-3]
\end{tikzcd}\]
namely that $f(a) = x g(a) x^{-1}$ $\forall a \in \hom(H, G) $. In other words, $g = x^{-1} f x$ which is identically the condition for $x$ to be a morphism in $\hqconj{\hom(H, G)}{G} $. Therefore the Hom-sets of the domain and codomain are in bijection, so this functor is fully faithful. Essential surjectivity is clear so they are categorically equivalent. 
\end{proof}
\begin{corollary}
For a group $G$ 
\[
\hom(S^1, \sB G) \simeq \hqconj{G}{G} 
\]
 \label{homotopyquotient}
\end{corollary}
\begin{proof}
    Since $\sB\ZZ \simeq S^1$ and $\hom(\ZZ, G) \simeq G$, an application of \cref{homHG} with $H:= \ZZ$ sees the result.
\end{proof}
The following fact about groupoids is the rest of what we need.
\begin{lemma}
    Given a groupoid $\cG$, there is an equivalence between groupoids:
    \[
    \cG \simeq \coprod_{[x]\in\pi_0\cG}\sB\Aut_{\cG}(x)~.
    \]
    \label{lemma}
\end{lemma}
\begin{proof}
    Choose for each $[x] \in \pi_0\cG$ a representative element $x \in [x]$. Consider the map
    \begin{equation}
    \coprod_{[x]\in\pi_0\cG}\sB\Aut_{\cG}(x) \to \cG~.
    \label{lemmafunctor}
    \end{equation}
    By the universal property of coproducts it is characterized by the functors that select the canonical action $\Aut_{\cG}(x) \curvearrowright x$ for each $x$:
    \[
    \sB \Aut_{\cG}(x) \xra{\langle \Aut_{\cG}(x) \curvearrowright x\rangle} \cG ~.
    \]
    This map is clearly essentially surjective. Namely, given an object $g$ of $\cG$, for $y \in [g]$ the chosen representative, then $g \simeq y$. 

    Additionally this map is fully faithful; note for the left side of \cref{lemmafunctor}:
    \[
    \hom_{\sf LHS}([x], [y]) = 
    \begin{cases} 
      \Aut_{\cG}(x)  & [x]=[y] \\
      \emptyset & \text{otherwise}
   \end{cases}
    \]
    and the right side
     \[
    \hom_{\sf RHS}(\Aut_{\cG}(x) \curvearrowright x, \Aut_{\cG}(y) \curvearrowright y) = 
    \begin{cases} 
      \Aut_{\cG}(x)  & x \simeq y \\
      \emptyset & \text{otherwise}
   \end{cases} ~.
    \]
    As $\cG$ is a groupoid, $x \simeq y$ if and only if $[x] = [y]$,
    \[
    \hom_{\sf RHS}(\Aut_{\cG}(x) \curvearrowright x, \Aut_{\cG}(y) \curvearrowright y) = \hom_{\sf LHS}([x], [y])~.
    \]
    Therefore this map is an equivalence of categories.
\end{proof}
\begin{observation}
      Let $G \curvearrowright X$ be a group acting on a set, then $\Aut_{X_{/\!/ G}}(x) = G_x$ where $G_x$ is the stabilizer subgroup of $x$ in G.
\end{observation}
This observation and $\pi_0\left(X_{/\! / G}\right) = X_{/G}$ leads to the following immediate corollary of \cref{lemma}.
\begin{corollary}
    Given $G$ a group and $X$ such that $G \curvearrowright X$, then 
    \[
    X_{/\! / G} \simeq \coprod_{[x]\in X_{/ G}} \sB G_x ~.
    \] 
    \label{cor}
\end{corollary}
Now we return to the proof of the main result.
\begin{proof}[Proof of \cref{MapS1}]
    The proof is an application of \cref{homotopyquotient} and \cref{cor}. The stabilizers are given by the centralizers since the action is by conjugation. Altogether
    \[
    \hom(S^1,\sB G) 
    \stackrel{(\ref{homotopyquotient})}{\simeq} \hqconj{G}{G}
    \stackrel{(\ref{cor})}{\simeq} \coprod_{[g] \in G_{/^{\sf conj} G}}\sB \sZ_{G}(g) ~.
    \]
\end{proof}
\section{Proof of Theorem}    
\test

\begin{proof}
    % Recall :
    % \[
    % \coprod_{r\geq 0} \coprod_{\substack{\sf r \hspace{.5ex} cycle \\ \sf type \hspace{.5ex} \omega}} \sB C_\omega \simeq HH\left(\coprod_{0 \leq r \leq N}\sB \Sigma_r\right)
    % \]
    % So it will suffice to show 
    % \[
    % HH(\Fin) \simeq  HH\left(\coprod_{0 \leq r \leq N}\sB \Sigma_r\right) 
    % \]

  First \cref{Reedyvaluewisecolim}, \cref{Unn=HH}, \cref{HH=fact} and that colimits commute with colimits implies: 
    \[
        {\sf HH}(\cR) \stackrel{\ref{HH=fact}}{\simeq} \int_{S^1} \cR 
        \stackrel{\ref{Reedyvaluewisecolim}}{\simeq} \int_{S^1} \underset{n \to \infty}\colim \cR_{\le n} 
        \label{line2}\\ 
        \simeq \underset{n \to \infty}\colim \int_{S^1}  \cR_{\le n} \quad  \\  
        \stackrel{\ref{Unn=HH}}{\simeq} \underset{n \to \infty}\colim |\Un_{\le n}|~.
    \]
    Therefore we must understand $\underset{n \to \infty}\colim |\Un_{\le n}|$ which is understanding $\Un_{\le n}$ by \cref{valuewisecolim}. 
    We seek out a description of $\Un_{\le n}$ through induction on $n$. Explicitly we aim to show that 
    \[|\Un_{\le n}| \simeq \coprod_{0 \ge r \ge } \coprod_{[r]\in\cR_{= n}} \coprod_{[\phi]\in \Aut_{\cR}(r)_{/^{\sf conj}}} \sB \sZ_{\Aut_{\cR}(r)}(\phi)~.\]
    
    We begin with a base case of $n=0$.
    %$n = \max \{{\sf Card} (\ell(x) | x \in \overline\lambda\}$
     First note that the morphisms in ${\cR_{\leq 0}^-}$ must be isomorphisms since non-isomorphisms lower degree and 0 is the lowest degree. This forces morphisms in $\cR_{\leq 0}^+$ to be isomorphisms since they map between objects of degree 0. Consequently, the Reedy factorization of any morphism $f \in {\cR_{\leq 0}}$ must be a composition of isomorphisms. Then $\cR_{\le 0}$ has only invertible morphisms so letting $[\cR_{\le 0}]^{\sf iso} := \pi_0(\Obj(\cR_{\leq 0}))$ we see by \cref{lemma}:
     \[ \cR_{\le 0} = \coprod_{[r]\in [\cR_{\leq 0}]^{\sf iso}} \sB\Aut_{\cR}(r)~.\]
     Thus, nonabelian Poincare duality (\cite{Ayala_2015} Corollary 4.6), and the connectedness of the circle imply
     \[
     \int_{S^1}\cR_{\le 0} \simeq \hom\left(S^1, \cR_{\le 0}\right)\simeq \hom\left(S^1, \coprod_{[r]\in [\cR_{\leq 0}]^{\sf iso}} \sB\Aut_{\cR}(r)\right) \simeq \coprod_{[r]\in[\cR_{\leq 0}]^{\sf iso}} \hom(S^1, \sB\Aut_{\cR}(r))~.
     \]  
     Finally recalling \cref{homcent} gives
     \[
     \hh(\cR) \simeq \int_{S^1}\cR_{\le 0} \simeq \coprod_{[r]\in [\cR_{\leq 0}]^{\sf iso}} \hom(S^1, \sB\Aut_{\cR}(r)) \simeq
     \coprod_{[r]\in [\cR_{\leq 0}]^{\sf iso}} \coprod_{[\phi]\in \Aut_{\cR}(r)_/{\sf conj}} \sB \sZ{\Aut_{\cR}(r)}(\phi)~,
     \]
     as desired.
   
    Now we're ready to begin the inductive step. First, we show $\Un_{\le n} \simeq \Un_{\le n-1} \amalg \Un_{= n}$. Using the factorization of morphisms in a Reedy category and \cref{paraobjects} we can consider two endofunctors on $\Un_{\leq n}$ 
    \[
    \sf sd:
     {\sf Un}_{\le n} 
     \longrightarrow 
     \Un_{\le n}
     \]
     \[
    \sf F:
     {\sf Un}_{\le n} 
     \longrightarrow 
     \Un_{\le n} 
     \]
     indicated pictorially on objects, as in \cref{Unaspic}, in \cref{3circles}.
    \begin{figure}
      % \documentclass{amsart}
% \usepackage{tikz,tkz-euclide}
% \begin{document}
\begin{tikzpicture}
    \node at (-4, 8) {$(\lambda, \ell)$};
    \node at (-2.5, 8) {=};
    \draw[<-] (-4, 3) to (-4, 7);
    \node at (-3.5, 5) {$T_{(\lambda, \ell)}$};
    %%% TOP CIRCLE %%%
    
    % define a circle (center=O and \radius)
      \coordinate (T) at (1,8);
      \def\radius{1.5cm}
    
      % draw this circle and its center
      \draw (T) circle[radius=\radius];
    
      % define a random point (A) on this circle
      \path (T) ++(0:\radius) coordinate (0);
      \path (T) ++(30:\radius) coordinate (30);
      \path (T) ++(60:\radius) coordinate (60);
      \path (T) ++(90:\radius) coordinate (90);
      \path (T) ++(105:\radius) coordinate (105);
      \path (T) ++(75:\radius) coordinate (75);
      \path (T) ++(45:\radius) coordinate (45);
      \path (T) ++(15:\radius) coordinate (15);
      \path (T) ++(5:\radius) coordinate (5);
      \path (T) ++(35:\radius) coordinate (35);
      \path (T) ++(65:\radius) coordinate (65);
      \path (T) ++(335:\radius) coordinate (335);
      \path (T) ++(345:\radius) coordinate (345);
      \path (T) ++(95:\radius) coordinate (95);
      \path (T) ++(225:\radius) coordinate (225);
      \path (T) ++(245:\radius) coordinate (245);
      \path (T) ++(205:\radius) coordinate (205);
      \path (T) ++(330:\radius) coordinate (330);
      \path (T) ++(120:\radius) coordinate (120);
      \path (T) ++(275:\radius) coordinate (275);
      \path (T) ++(185:\radius) coordinate (185);
      \path (T) ++(300:\radius) coordinate (300);
    
      % draw (A) with a label
      \fill[red] (30) circle[radius=2pt] ++(30:1em) node {$X_0$};
      \fill[red] (90) circle[radius=2pt] ++(90:1em) node {$X_L$};
      \fill[red] (330) circle[radius=2pt] ++(330:1em) node {$X_1$};
      
      \fill[black] (65) ++(65:2em) node {$f_L$};
      \fill[black] (0) ++(0:2em) node {$f_0$};
      \fill[black] (120) ++(120:2em) node {$f_{L-1}$};
      \fill[black] (300) ++(300:2em) node {$f_1$};
    
      \fill[black] (225) ++(225:2em) circle[radius=1pt] node {};
      \fill[black] (185) ++(185:2em) circle[radius=1pt] node {};
      \fill[black] (205) ++(205:2em) circle[radius=1pt] node {};
    
      \draw[<-] (335) ++(365:1em) arc[radius=\radius, start angle=335, end angle=385];
      \draw[<-] (35) ++(65:1em) arc[radius=\radius, start angle=35, end angle=85];
      \draw[<-] (95) ++(125:1em) arc[radius=\radius, start angle=95, end angle=145];
      \draw[<-] (275) ++(305:1em) arc[radius=\radius, start angle=275, end angle=325];
    
      \draw[<-] (1, 4.5) to (1, 5.5);
      \node at (1.5, 5) {$T_{(\lambda, \ell)}$};
    %%%% MIDDLE CIRCLE %%%%

    \node at (-4, 2) {$\sf sd (\lambda, \ell)$};
    \node at (-2.5, 2) {=};
    
      % define a circle (center=O and \radius)
      \coordinate (O) at (1,2);
      \def\radius{1.5cm}
    
      % draw this circle and its center
      \draw (O) circle[radius=\radius];
    
      % define a random point (A) on this circle
      \path (O) ++(0:\radius) coordinate (0);
      \path (O) ++(30:\radius) coordinate (30);
      \path (O) ++(60:\radius) coordinate (60);
      \path (O) ++(90:\radius) coordinate (90);
      \path (O) ++(105:\radius) coordinate (105);
      \path (O) ++(75:\radius) coordinate (75);
      \path (O) ++(45:\radius) coordinate (45);
      \path (O) ++(15:\radius) coordinate (15);
      \path (O) ++(5:\radius) coordinate (5);
      \path (O) ++(35:\radius) coordinate (35);
      \path (O) ++(65:\radius) coordinate (65);
      \path (O) ++(335:\radius) coordinate (335);
      \path (O) ++(345:\radius) coordinate (345);
      \path (O) ++(95:\radius) coordinate (95);
      \path (O) ++(225:\radius) coordinate (225);
      \path (O) ++(245:\radius) coordinate (245);
      \path (O) ++(205:\radius) coordinate (205);
      \path (O) ++(330:\radius) coordinate (330);
      \path (O) ++(305:\radius) coordinate (305);
      \path (O) ++(315:\radius) coordinate (315);
    
      % draw (A) with a label
      \fill[red] (0) circle[radius=2pt] ++(0:2em) node {$\im(f_0)$};
      \fill[red] (30) circle[radius=2pt] ++(30:1em) node {$X_0$};
      \fill[red] (60) circle[radius=2pt] ++(30:2em) node {$\im(f_L)$};
      \fill[red] (90) circle[radius=2pt] ++(90:1em) node {$X_L$};
      \fill[red] (330) circle[radius=2pt] ++(330:1em) node {$X_1$};
      
      \fill[black] (15) ++(15:2em) node {$f^-_0$};
      \fill[black] (45) ++(20:2em) node {$f^+_0$};
      \fill[black] (75) ++(75:2em) node {$f^-_L$};
      \fill[black] (345) ++(345:2em) node {$f^+_1$};
      \fill[black] (105) ++(105:2em) node {$f^+_L$};
      \fill[black] (315) ++(315:2em) node {$f^-_1$};
    
      \fill[black] (225) ++(225:2em) circle[radius=1pt] node {};
      \fill[black] (245) ++(245:2em) circle[radius=1pt] node {};
      \fill[black] (205) ++(205:2em) circle[radius=1pt] node {};
    
      \draw[<-] (335) ++(345:1em) arc[radius=\radius, start angle=335, end angle=355];
      \draw[<-] (5) ++(5:1em) arc[radius=\radius, start angle=5, end angle=25];
      \draw[<-] (35) ++(45:1em) arc[radius=\radius, start angle=35, end angle=55];
      \draw[<-] (65) ++(65:1em) arc[radius=\radius, start angle=65, end angle=85];
      \draw[<-] (95) ++(105:1em) arc[radius=\radius, start angle=95, end angle=115];
      \draw[<-] (305) ++(315:1em) arc[radius=\radius, start angle=305, end angle=325];

      \draw[<-] (1, -0.5) to (1, -1.5);
      \node at (1.5, -1) {$S_{(\lambda, \ell)}$};
    %%%%%%%% BOTTOM CIRCLE %%%%%%%%%%%

    \node at (-4, -4) {$F (\lambda, \ell)$};
    \node at (-2.5, -4) {=};
    \draw[<-] (-4, 1) to (-4, -3);
    \node at (-3.5, -1) {$S_{(\lambda,\ell)}$};

        % define a circle (center=O and \radius)
      \coordinate (B) at (1,-4);
      \def\radius{1.5cm}
    
      % draw this circle and its center
      \draw (B) circle[radius=\radius];
    
      % define a random point (A) on this circle
      \path (B) ++(0:\radius) coordinate (0);
      \path (B) ++(30:\radius) coordinate (30);
      \path (B) ++(60:\radius) coordinate (60);
      \path (B) ++(90:\radius) coordinate (90);
      \path (B) ++(105:\radius) coordinate (105);
      \path (B) ++(75:\radius) coordinate (75);
      \path (B) ++(45:\radius) coordinate (45);
      \path (B) ++(15:\radius) coordinate (15);
      \path (B) ++(5:\radius) coordinate (5);
      \path (B) ++(35:\radius) coordinate (35);
      \path (B) ++(65:\radius) coordinate (65);
      \path (B) ++(335:\radius) coordinate (335);
      \path (B) ++(345:\radius) coordinate (345);
      \path (B) ++(95:\radius) coordinate (95);
      \path (B) ++(225:\radius) coordinate (225);
      \path (B) ++(245:\radius) coordinate (245);
      \path (B) ++(205:\radius) coordinate (205);
      \path (B) ++(330:\radius) coordinate (330);
      \path (B) ++(120:\radius) coordinate (120);
      \path (B) ++(275:\radius) coordinate (275);
      \path (B) ++(185:\radius) coordinate (185);
      \path (B) ++(300:\radius) coordinate (300);
      \path (B) ++(355:\radius) coordinate (355);
    
      % draw (A) with a label
      \fill[red] (0) circle[radius=2pt] ++(0:2em) node {$\im(f_0)$};
      \fill[red] (60) circle[radius=2pt] ++(30:2em) node {$\im(f_L)$};
      %\fill[red] (330) circle[radius=2pt] ++(330:1em) node {$X_2$};
      
      \fill[black] (30) ++(15:3em) node {$f^+_0 \circ f^-_0 $};
      \fill[black] (90) ++(90:2em) node {$f^+_L \circ f^-_L $};
      \fill[black] (330) ++(345:3em) node {$f^+_1 \circ f^-_1 $};
    
      \fill[black] (225) ++(225:2em) circle[radius=1pt] node {};
      \fill[black] (185) ++(185:2em) circle[radius=1pt] node {};
      \fill[black] (205) ++(205:2em) circle[radius=1pt] node {};
    
      \draw[<-] (5) ++(35:1em) arc[radius=\radius, start angle=5, end angle=55];
      \draw[<-] (65) ++(95:1em) arc[radius=\radius, start angle=65, end angle=115];
      \draw[->] (355) ++(325:1em) arc[radius=\radius, start angle=355, end angle=305];

      \draw[<-] (1, 4.5) to (1, 5.5);
  \end{tikzpicture}

% \end{document}
      \caption{Pictorial definition on objects of two functors and two natural inclusions}
      \label{3circles}
    \end{figure}
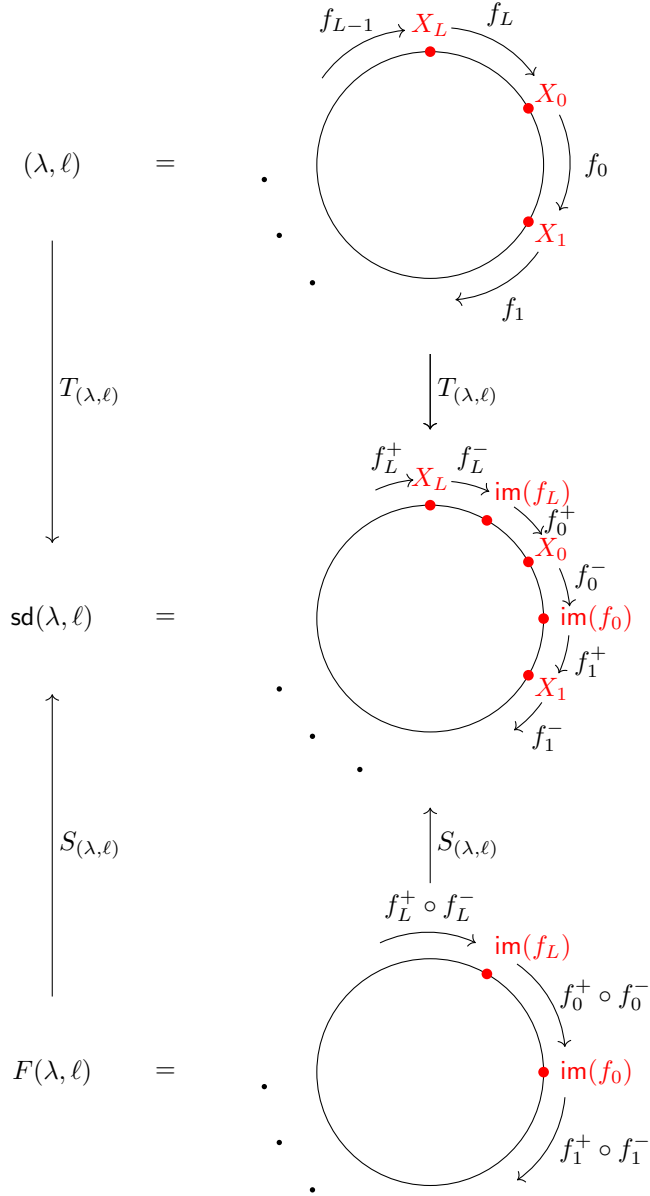 
    
    We now give an explicit construction of these functors. We begin by recalling \cref{paraobjects}: that the full subcategory of $\para$ consisting of objects of the form $\frac{1}{L} \ZZ$ is a skeleton. Thus any functor defined on $\Un_{\le n}$ need only be specified on elements of the form $\left(\frac{1}{L} \ZZ, \overline{\frac{1}{L} \ZZ} \xra{\ell}\cR_{\le n}\right)$. Recall $\overline{\frac{1}{L} \ZZ}$ is the quiver with objects $\{\frac{k}{L}+\ZZ\}_{0\le k < L}$ and morphisms freely generated by  $\frac{k}{L}+ \ZZ \to \frac{k+1}{L}+ \ZZ$ for each $0\le k < L$. Then $\ell$ is specified by for each $k \in \ZZ_{/L\ZZ}$ an object of $\cR$ with degree less than n, say $X_k := \ell(\frac{k}{L}+\ZZ)$, and a morphism $ f_k:X_k \to X_{k+1}$. Now to define $\sf sd$:
    \[
    \sd\left(\frac{k}{L} \ZZ, \ell\right) 
    := \left(\frac{k}{2L} \ZZ , \overline{\frac{k}{2L}\ZZ}\xra{\ell_{\sd}}\cR_{\le n}\right)
    \]
    where $\ell_{\sd}$ is defined on objects by for $p \in \ZZ_{/2L \ZZ} $
    \[
    \ell_{\sd}\left(\frac{p}{2L}+\mathbb{Z}\right) := 
    \begin{cases} 
      X_{p/2} & \text{p even} \\
      {\sf im}\left(f_{\frac{p-1}{2}}\right) & \text{p odd} ~,
   \end{cases}
    \]
    and, with a slight abuse of notation, on morphisms by
    \[
    \ell_{\sd}\left(\frac{p}{2L}+ \ZZ \to \frac{p+1}{2L} +\ZZ\right) := 
    \begin{cases} 
      f^-_{p/2}: X_{p/2} \to {\sf im}(f_{p/2}) & \text{p even} \\
      f^+_{\frac{p+1}{2}}: {\sf im}\left(f_{\frac{p-1}{2}}\right) \to X_{\frac{p+1}{2}} & \text{p odd} ~,
   \end{cases}
    \]
   where the indices are taken modulo $L$. 
   % Putting this all together we could write:
   %   \[
   %   \left(\frac{1}{L}\mathbb{Z} , \LARGE{\substack{\frac{m}{L}+\mathbb{Z} \mapsto X_m \\ \left(\frac{m}{L}+\mathbb{Z} \to \frac{m+1}{L}+\mathbb{Z}\right) \mapsto f_m}} \right)
   %   \xmapsto{\sd}
   %   \left(\frac{1}{2L}\mathbb{Z} , \LARGE{\substack{\normalsize{\frac{k}{2L}+\mathbb{Z} \mapsto}
   %   \begin{cases} 
   %     X_{k/2} & \text{k even} \\
   %     {\sf im}(f_{\frac{k-1}{2}}) & \text{k odd}
   % \end{cases}
   %   \\
   %   \normalsize{\left(\frac{k}{2L}+\mathbb{Z} \to \frac{k+1}{2L}+\mathbb{Z}\right) \mapsto}
   %   \begin{cases} 
   %    f^-_k & \text{k even} \\
   %    f^+_k & \text{k odd}
   % \end{cases}
   %    }} \right)
   %   ~.
   %   \]
     Similarly we can define a map $F$ by 
     \[
    F\left(\frac{k}{L}\ZZ , \ell\right) 
    := \left(\frac{k}{L}\ZZ +\frac{1}{2L} 
    , \overline{\frac{2k+1}{2L}\ZZ}\xra{\ell_F}\cR_{\le n}\right)
    \]
    where $\ell_{F}$ is defined on objects by
    \[
    \ell_{F}\left(\frac{k}{L}+\mathbb{Z}\right) := \sf{im}(f_k)~,
    \]
    and on morphisms by
    \[
    \ell_{F}\left( \frac{k}{L}+ \ZZ \to \frac{k+1}{L} +\ZZ\right) 
    := f^+_{k+1} \circ f^-_{k+1}   ~. 
    \]
    % As before, putting this altogether we see:
    %   \[
    %  \left(\frac{1}{L}\mathbb{Z} , \LARGE{\substack{\frac{m}{L}+\mathbb{Z} \mapsto X_m \\ \left(\frac{m}{L}+\mathbb{Z} \to \frac{m+1}{L}+\mathbb{Z}\right) \mapsto f_m}} \right)
    %  \xmapsto{F}
    %  \left(\frac{1}{L}\mathbb{Z}+\frac{1}{2L} , \LARGE{\substack{\frac{k}{L}+\mathbb{Z} \mapsto \sf{im}(f_k) \\ \left(\frac{k}{L}+\mathbb{Z} \to \frac{k+1}{L}+\mathbb{Z}\right) \mapsto f^+_k \circ f^-_k}} \right)
    %  ~.
    %  \]

    We now describe where these functors take morphisms. Towards this, consider $\lambda, \lambda' \in \para$ where $\lambda$ can be represented by $\frac{1}{L}\ZZ$ and $\lambda'$ by $\frac{1}{M}\ZZ$. Additionally define $\ell: \overline{\lambda} \to \cR_{\le n}$ and $\ell': \overline{\lambda'} \to \cR_{\le n}$ on both objects and morphisms by:
    \begin{align*}
        \ell_{\sf obj}:\frac{k}{L}+\ZZ &\mapsto X_k \\
        \ell_{\sf mor}:\left(\frac{k}{L}+\ZZ \to \frac{k+1}{L} + \ZZ\right) &\mapsto \left[f_k: X_k \to X_{k+1} \right]     
    \end{align*}
    and 
    \begin{align*}
        \ell'_{\sf obj}:\frac{j}{M}+\ZZ &\mapsto Y_j \\
        \ell'_{\sf mor}:\left(\frac{j}{M}+\ZZ \to \frac{j+1}{M}+\ZZ\right) &\mapsto \left[h_j: Y_j \to Y_{j+1}\right]     
    \end{align*}
    respectively. 
    For all morphisms $g:\lambda\to \lambda'$ in $\para$ the functor $\sd$ is defined as follows:
    \begin{equation}
    \begin{tikzcd}
	{\overline\lambda} &[-1ex]&[-1ex] {\overline{\lambda'}} &&& {\overline{\sd(\lambda)}} &[-2ex]&[-2ex] {\overline{\sd(\lambda')}} \\[-1em]
	& {\hspace{2ex} \rotatebox[origin=c]{30}{$\implies$} \eta \hspace{2ex}} & {} &&& {} & {\rotatebox[origin=c]{30}{$\implies$}  { \sf sd}(\eta)} \\
	& \cR &&&&& \cR
	\arrow["{{{\overline g}}}", from=1-1, to=1-3]
	\arrow["\ell"', from=1-1, to=3-2]
	\arrow["{{{\ell'}}}", from=1-3, to=3-2]
	\arrow["{\overline{\sd(g)}}", from=1-6, to=1-8]
	\arrow["{\ell_{\sd}}"', from=1-6, to=3-7]
	\arrow["{\ell'_{\sd}}", from=1-8, to=3-7]
	\arrow["\sd", maps to, from=2-3, to=2-6] 
\end{tikzcd}.
\label{triangle}
\end{equation}
    Here, $\sd(g)$ is defined via:
    \[
    \sd (g) : 
    \normalsize{\frac{k}{2L}  \mapsto}
     \begin{cases} 
      g(\frac{k/2}{L})  & \text{k even} \\
      g(\frac{(k-1)/2}{L}) + \frac{1}{2M} & \text{k odd}
   \end{cases}~.
    \]
    Note ${\sf sd}(g)$ does take values in $\overline{\sd(\lambda')} \simeq \frac{1}{2M}\ZZ$ as desired. For ease of notation, we write $g(k) := p$ instead of $g\left(\frac{k}{L}\right) = \frac{p}{M}$. 
    
    The natural transformation, ${\sf sd}(\eta)$, extends $\eta$. Namely $\eta$ necessitates the existence of the following commutative diagram for all $k$ and for all morphisms $q$ in $\para$: 
    \[\begin{tikzcd}
	{\ell(\frac{k}{L}+\ZZ)} && {\ell(\frac{k+1}{L}+\ZZ)} && {X_k} && {X_{k+1}} \\
	& \circlearrowleft && {=} && \circlearrowleft \\
	{\ell'\circ\overline g(\frac{k}{L}+\ZZ)} && {\ell'\circ\overline g(\frac{k+1}{L}+\ZZ)} && {Y_{g(k)}} && {Y_{g(k+1)}}
	\arrow["{{\ell(q)}}", from=1-1, to=1-3]
	\arrow["{{\eta_k}}"', from=1-1, to=3-1]
	\arrow["{{\eta_{k+1}}}", from=1-3, to=3-3]
	\arrow["{\ell(q)}", from=1-5, to=1-7]
	\arrow["{{\eta_k}}"', from=1-5, to=3-5]
	\arrow["{{\eta_{k+1}}}", from=1-7, to=3-7]
	\arrow["{{\ell'\circ\overline g(q)}}"', from=3-1, to=3-3]
	\arrow["{{\ell'\circ\overline g(q)}}"', from=3-5, to=3-7]
\end{tikzcd}\hspace{2ex} .\]
    Then for ${\sf sd}(\eta)$ to be a natural transformation in \cref{triangle}, we would need

    \[\begin{tikzcd}
	{\ell_{\sd}(\frac{k}{2L}+\ZZ)} && {\ell_{\sd}(\frac{k+1}{2L}+\ZZ)} \\
	& \circlearrowleft \\
	{[\ell'_{\sd}\circ\overline{\sd(g)}](\frac{k}{2L}+\ZZ)} && {[\ell'_{\sd}\circ\overline{\sd(g)}](\frac{k+1}{2L}+\ZZ)}
	\arrow["{{\ell_{\sd}(h)}}", from=1-1, to=1-3]
	\arrow["{{\sd}({\eta_k})}"', from=1-1, to=3-1]
	\arrow["{{\sd}({\eta_{k+1}})}", from=1-3, to=3-3]
	\arrow["{{\ell'_{\sd}\circ\overline{\sd(g(h))}}}"', from=3-1, to=3-3]
    \end{tikzcd}~,\]
    to commute for all k. This equates to, for all $k$ even, a factorization: 
    \begin{equation}
    \begin{tikzcd}
	{X_{k/2}} && {{\sf im}(f_{k/2})} \\
	& \circlearrowleft \\
	{Y_{g(k/2)}} && {{\sf im}(h_{g(k/2)})}
	\arrow[two heads,"{}", from=1-1, to=1-3]
	\arrow["{\eta_{k/2}}"', from=1-1, to=3-1]
	\arrow["{}", dashed, from=1-3, to=3-3]
	\arrow[two heads,"{}"', from=3-1, to=3-3]
    \end{tikzcd}~;
    \label{square1}
    \end{equation}
    and, for every $k$ odd, a factorization:
    \begin{equation}
    \begin{tikzcd}
	{{\sf im}(f_{(k-1)/2})} && {X_{(k+1)/2}} \\
	& \circlearrowleft \\
	{{\sf im}(h_{g((k-1)/2)})} && {Y_{g((k+1)/2)}}
	\arrow[hook,"{}", from=1-1, to=1-3]
	\arrow["{}"', dashed, from=1-1, to=3-1]
	\arrow["{\eta_{(k+1)/2}}", from=1-3, to=3-3]
	\arrow[hook,"{}"', from=3-1, to=3-3]
    \end{tikzcd}~.
    \label{square2} 
    \end{equation}
    All these squares, together, form the diagram shape from \cref{imcommute}:
    \[\begin{tikzcd}
	{X_{k/2}} && {X_{(k/2)+1}} \\
	& {{\sf im}(f_{k/2})} \\
	{Y_{g(k/2)}} && {Y_{g((k/2)+1)}} \\
	& {{\sf im}(h_{g(k/2)})}
	\arrow["{\ell(q)}", from=1-1, to=1-3]
	\arrow[two heads, from=1-1, to=2-2]
	\arrow["{\eta_k}", from=1-1, to=3-1]
	\arrow["{{\eta_{k+1}}}", from=1-3, to=3-3]
	\arrow[hook, from=2-2, to=1-3]
	\arrow[dashed, from=2-2, to=4-2]
	\arrow["{[{\ell'\circ\overline{g}](q)}}"{description, pos=0.7}, from=3-1, to=3-3]
	\arrow[two heads, from=3-1, to=4-2]
	\arrow[hook, from=4-2, to=3-3]
\end{tikzcd}\hspace{2ex} .\]
    Therefore, by \cref{imcommute} the necessary map exists and there are factorizations as in \cref{square1} and \cref{square2}. Then ${\sf sd}(\eta)$ is indeed a natural transformation in \cref{triangle}. So $\sd$ is a functor. 
    
    This process can be repeated for $F$, in fact it is significantly simpler:
    \[
    F (g) : 
    \frac{2k+1}{2L}  \mapsto
      g\left(\frac{k}{L}\right)+ \frac{1}{2M}~.
    \]

    Finally, we look to define the natural inclusions as natural transformations. These can be defined abstractly by recalling $\Un_{\le n} \to \para$ is a Cartesian fibration, so there exists final fillers

\[\begin{tikzcd}
	{*} &&& {\Un_{\le n}} &&&& {*} &&& {\Un_{\le n}} \\
	\\
	{[1]} &&& \para &&&& {[1]} &&& \para
	\arrow["{{\langle \sf sd(\lambda, \ell)\rangle}}", from=1-1, to=1-4]
	\arrow["{{\langle 1 \rangle }}"', from=1-1, to=3-1]
	\arrow["{\sf fgt}", from=1-4, to=3-4]
	\arrow["{{\langle {\sf sd}(\lambda, \ell)\rangle }}", from=1-8, to=1-11]
	\arrow["{{\langle 1 \rangle }}"', from=1-8, to=3-8]
	\arrow["{\sf fgt}", from=1-11, to=3-11]
	\arrow[dashed, from=3-1, to=1-4]
	\arrow["{{\langle \lambda \to \sf sd(\lambda) \rangle}}"', from=3-1, to=3-4]
	\arrow[dashed, from=3-8, to=1-11]
	\arrow["{{\langle F(\lambda) \to {\sf sd}(\lambda) \rangle}}"', from=3-8, to=3-11]
\end{tikzcd} .\]
    Denote these fillers by $\langle T \rangle$ and respectively $\langle S \rangle$. More explicitly T is:
   
    \[
    T : \id \to \sd \\
    \]
    on objects and
    \[
    T_{\left(\frac{1}{L}\ZZ, \ell\right)} : 
    \left(\frac{1}{L}\ZZ, \ell\right) \to
    \sd\left(\frac{1}{L} \ZZ, \ell\right) ~.
    \]
    on morphisms.
    This is a morphism in $\Un_{\le n}$ and therefore can be specified by a morphism in $\para$ and a communtative triangle. Then $T$ is given by
    \begin{align*}
    T &:  \frac{1}{L}\ZZ \to \frac{1}{2L}\ZZ \\
    T &:  \frac{k}{L}   \mapsto  \frac{2k}{2L} ,
    \end{align*}
    and 
    \[\begin{tikzcd}
	{\overline\lambda} &[-1em]&[-1em] {\overline{\sd(\lambda)}} \\[-1em]
	& \rotatebox[origin=c]{30}{$\circlearrowleft$}  \\
	& \cR
	\arrow["{\overline T}", from=1-1, to=1-3]
	\arrow["\ell"', from=1-1, to=3-2]
	\arrow["\ell_{\sd}", from=1-3, to=3-2]
    \end{tikzcd}~.\]
    We define $S \in \Un_{\le n}$ similiarly:
    \begin{align*}
    S &: F \to \sd \\
    S_{\left(\frac{1}{L} \ZZ, \ell\right)} &:
     F\left(\frac{1}{L} \ZZ, \ell\right) 
     \to  \sd\left(\frac{1}{L} \ZZ, \ell\right)
    \end{align*}
    with the morphism 
    \begin{align*}
    S &:  \frac{1}{L}\ZZ \to \frac{1}{2L}\ZZ \\
    S &:  \frac{k}{L}   \mapsto  \frac{2k+1}{2L} ,
    \end{align*}
    and the communative triangle 
    \[\begin{tikzcd}
	{\overline{F(\lambda)}} &[-1em]&[-1em] {\overline{\sd(\lambda)}} \\[-1em]
	& \rotatebox[origin=c]{30}{$\circlearrowleft$}  \\
	& \cR
	\arrow["{\overline S}", from=1-1, to=1-3]
	\arrow["\ell_F"', from=1-1, to=3-2]
	\arrow["{\ell_{\sd}}", from=1-3, to=3-2]
    \end{tikzcd}~.\]
     Together these natural transformations establish the diagram 
    \[\begin{tikzcd}
	{\Un_{\le n}} && \begin{array}{c} {\substack{\hspace{2ex} \rotatebox[origin=c]{270}{$\Rightarrow$} \hspace{.75ex} T \\ \\ \\ \hspace{2ex} \rotatebox[origin=c]{90}{$\Rightarrow$} \hspace{.75ex}S}} \end{array} && {\Un_{\le n}} 
	\arrow["{F}"{description}, curve={height=24pt}, from=1-1, to=1-5]
	\arrow["{\id}"{description}, curve={height=-24pt}, from=1-1, to=1-5]
	\arrow["{\sd}"{description}, from=1-1, to=1-5]
    \end{tikzcd}~.\]

    Taking $\infty$-groupiod completions, also called geometric realizations, determines a (homotopy) commutative diagram among spaces: 
    \[\begin{tikzcd}
	{|\Un_{\le n}|} && \begin{array}{c} {\substack{\hspace{5ex} \rotatebox[origin=c]{270}{$\Rightarrow$} \text{ (htpy)} \\ \\ \\ \hspace{5ex} \rotatebox[origin=c]{90}{$\Rightarrow$} \text{ (htpy)}}} \end{array} && {|\Un_{\le n}|}
	\arrow["{|F|}"{description}, curve={height=24pt}, from=1-1, to=1-5]
	\arrow["{|\id|}"{description}, curve={height=-24pt}, from=1-1, to=1-5]
	\arrow["{|\sd|}"{description}, from=1-1, to=1-5]
    \end{tikzcd}~.\]
    Then $|F|$ and $|\sd|$ are homotopic to the identity and so is $|F^r|$ for $r \in \ZZ_+$.  We have now reduced showing $|\Un_{\le n}| \simeq |\Un_{\le n-1}|\amalg |U|$ for some $U$, to $|F^r(\Un_{\le n})| \simeq |\Un_{\le n-1}|\amalg |U|$
    
    Before we can show $F^r$ does take values in the desired disjoint union, we consider the full subcategory $(\para^{\op})_{\leq N} \subset \para^{\op}$ with elements isomorphic to $\frac{1}{L}\ZZ$ where $L \le N$, i.e. $\lambda$ such that the quiver $\ov{\lambda}$ has at most $N$ objects.
    Then $\underset{N \to \infty}\colim (\para^{\op})_{\leq N} = \para^{\op}$ so 
    \begin{align*}
        \Un_{\le n} 
        &\simeq  \Un \left( \underset{N \to \infty}\colim[(\para^{\op})_{\leq N}] \to {\sf Quiv}^{\sf op} \xrightarrow{\hom_\Cat(\_, \cR_{\le n})} \spaces \right) \\
        &\simeq \underset{N \to \infty}\colim \left[ \Un \left( (\para^{\op})_{\leq N} \to {\sf Quiv}^{\sf op} \xrightarrow{\hom_\Cat(\_, \cR_{\le n})} \spaces \right) \right]  \stepcounter{equation}\tag{\theequation} \label{unline2} \\
    \end{align*}  
    where \cref{unline2} is due to Cartesian fibrations being exponentiable fibrations as shown in Lemma 2.15 in \cite{ayala2020fibrationsinftycategories}. We will use the notation
    \[
     \Un_{\le n, \le N} :=  \left[ \Un \left( (\para^{\op})_{\leq N} \to {\sf Quiv}^{\sf op} \xrightarrow{\hom_\Cat(\_, \cR_{\le n})} \spaces \right) \right] ~.
        \stepcounter{equation}\tag{\theequation} \label{colimNUn}
    \]
    
    Now we induct on N to show $F^k$ takes values in $ \Un \left( \para^{\sf op} \to {\sf Quiv}^{\sf op} \xrightarrow{\hom_\Cat(\_, \cR_{=n})} \spaces \right)  \amalg \Un_{\le n-1}$ for some large enough $k$. For the sake of readability we denote
    \[
    \Un_{=n}:= \Un \left( \para^{\sf op} \to {\sf Quiv}^{\sf op} \xrightarrow{\hom_\Cat(\_, \sB \cR_{=n})} \spaces \right)~.
    \]
    Note that for $(\lambda, \ell) \in \Un_{=n}$, then $\ell$ either takes objects to objects of degree strictly less than $n$ or morphisms to isomorphisms
    
    First the base case: let $(\lambda, \ell) \in \Un_{\le n, \le 1}$. The unique vertex is labelled as an object with degree $n$ or with degree less than $n$. If the degree is less then $n$, than $(\lambda, \ell) \in \Un_{\le n-1}$ so we are done. If the degree is equal to $n$, either the labeled morphism is an isomorphism, in which case  $(\lambda, \ell) \in  \Un_{=n}$, or the labelled morphism is not an isomorphism in which case the degree of the image of the morphism is less than $n$ so $F(\lambda, \ell)\in \Un_{\le n-1}$.

    % {\color{magenta}
    % introduce the full subcategory $(\para^{\op})_{\leq N} \subset \para^{\op}$ consisting of those $\lambda$ such that the quiver $\Obj(\ov{\lambda}))$ has at most $N$ objects/vertices.  

    % use this to define a full subcategory ${\sf Un}_{\leq n, \leq N} \subset {\sf Un}_{\leq n}$.

    % observe that $\colim_N (\para^{\op})_{\leq N} = \para^{\op}$.
    % conclude that $\colim_N {\sf Un}_{\leq n, \leq N} = {\sf Un}_{\leq n}$.
    % use that $|-|$ commutes with filtered colimits.

    % then, in this induction step (on $n$), use induction on $N$.
    % For this, your argument below is valid.
    % }
        Now we proceed to the induction step. Let $(\lambda, \ell) \in \Un_{\le n, \le N}$. Let $L \in \NN$ be such that $\lambda \simeq \frac{1}{L}\ZZ \in \para$. Denote $X_n := \ell\left(\frac{n}{L}+\ZZ\right)$ and $f_n := \ell\left(\frac{n}{L}+\ZZ \to \frac{n+1}{L}+\ZZ \right)$. Note, if $L < N$ we're done by induction and if $L=1$ we are done by the base case, so suppose $L=N>1$. 
        Now suppose there exists $i$ such that $|{\sf im}(f_i)| \le n$. Without loss of generality let $i=1$. Then both $\deg( {\sf im}(f^+_1\circ  f^-_1))$ and $\deg({\sf im}(f^+_2\circ f^-_2))$ must be less than $n$, by \cref{factor}. This application of \cref{factor} is depicted by \cref{smallfig}. Each application of $F$ causes one additional vertex to be mapped to an object with degree less than $n$. This process can be repeated until all of the vertices are mapped to objects with degree less than $n$. Explicitly, 
    \begin{figure}[h]   
         \[
    \begin{tikzpicture}

    % define a circle (center=O and \radius)
      \coordinate (T) at (-3.5,0);
      \coordinate (B) at (3.5,0);
      \def\radius{1cm}
    
      % draw this circle and its center
      \draw (T) circle[radius=\radius];
      \draw (B) circle[radius=\radius];

      % Arrows
      \draw[|->] (-1.5,0) to (1.5,0);
      \node at (0, .25) {\sf sd};
    
      % define a random point (A) on this circle
      \path (T) ++(0:\radius) coordinate (0);
      \path (T) ++(120:\radius) coordinate (120);
      \path (T) ++(240:\radius) coordinate (240);
      \path (T) ++(60:\radius) coordinate (60);
      \path (T) ++(180:\radius) coordinate (180);
      \path (T) ++(310:\radius) coordinate (310);
      \path (T) ++(10:\radius) coordinate (10);
      \path (T) ++(130:\radius) coordinate (130);
      \path (T) ++(250:\radius) coordinate (250);
      \path (T) ++(200:\radius) coordinate (200);

      \fill[blue] (0) circle[radius=2pt] ++(0:.75em) node {$X_2$};
      \fill (120) circle[radius=2pt] ++(120:.75em) node {$X_1$};
      \fill (240) circle[radius=2pt] ++(240:.75em) node {$X_3$};

      \fill[black] (60) ++(60:1em) node {$f_1$};
      \fill[black] (180) ++(180:1em) node {$\vdots$};
      \fill[black] (310) ++(310:1em) node {$f_2$};

      \draw[<-] (10) ++(60:.5em) arc[radius=\radius, start angle=10, end angle=110];
      \draw[<-] (200) ++(180:.5em) arc[radius=\radius, start angle=200, end angle=230];
      \draw[<-] (130) ++(180:.5em) arc[radius=\radius, start angle=130, end angle=160];
      \draw[<-] (250) ++(300:.5em) arc[radius=\radius, start angle=250, end angle=350];

      % define a random point (A) on this circle
      \path (B) ++(0:\radius) coordinate (0);
      \path (B) ++(120:\radius) coordinate (120);
      \path (B) ++(240:\radius) coordinate (240);
      \path (B) ++(60:\radius) coordinate (60);
      \path (B) ++(180:\radius) coordinate (180);
      \path (B) ++(310:\radius) coordinate (310);
      \path (B) ++(10:\radius) coordinate (10);
      \path (B) ++(130:\radius) coordinate (130);
      \path (B) ++(250:\radius) coordinate (250);
      \path (B) ++(200:\radius) coordinate (200);
      \path (B) ++(70:\radius) coordinate (70);
      \path (B) ++(90:\radius) coordinate (90);
       \path (B) ++(330:\radius) coordinate (330);

      \fill[blue] (0) circle[radius=2pt] ++(0:.75em) node {$X_2$};
      \fill (120) circle[radius=2pt] ++(120:.75em) node {$X_1$};
      \fill (240) circle[radius=2pt] ++(240:.75em) node {$X_3$};
      \fill[blue] (60) circle[radius=2pt] ++(0:1.75em) node {${\sf im}(f_1)$};
      %\fill[red] (180) circle[radius=2pt] ++(120:.75em) node {$Y$};
      \fill[blue] (310) circle[radius=2pt] ++(240:.5em) node {$\hspace{8ex} {\sf im}(f_2)$};

      \fill[black] (10) ++(60:1.15em) node {$ \hspace{2ex}f_1^+$};
      \fill[black] (90) ++(90:1.25em) node {$f_1^-$};
      \fill[black] (180) ++(180:1em) node {$\vdots$};
      \fill[black] (250) ++(310:1.5 em) node {$f_2^+$};
      \fill[black] (330) ++(0:1.35 em) node {$f_2^-$};

      \draw[<-] (10) ++(60:.5em) arc[radius=\radius, start angle=10, end angle=30];
      \draw[<-] (70) ++(60:.5em) arc[radius=\radius, start angle=70, end angle=110];
      \draw[<-] (200) ++(180:.5em) arc[radius=\radius, start angle=200, end angle=230];
      \draw[<-] (130) ++(180:.5em) arc[radius=\radius, start angle=130, end angle=160];
      \draw[<-] (250) ++(300:.5em) arc[radius=\radius, start angle=250, end angle=290]; 
      \draw[<-] (330) ++(300:.5em) arc[radius=\radius, start angle=330, end angle=350]; 
    \end{tikzpicture}   
    .\]
        \caption{Given that $X_2$ is degree less than $n$, then both ${\sf im}(f)$ and ${\sf im}(g)$ must also have degree less than $n$ as $f^+$ is increasing and $f^-$ is decreasing. Here the blue denotes that the degree of the object is less than $n$. }
        \label{smallfig}
    \end{figure}
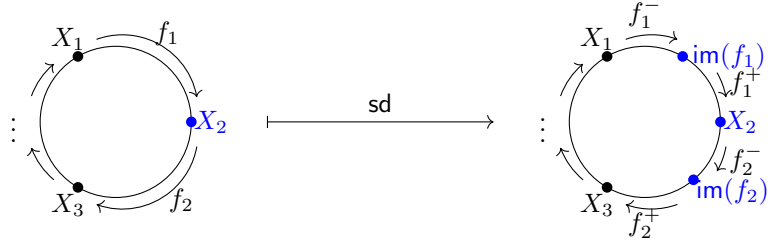  
        Therefore by repeating $F$ a maximum of $N$ times, all vertices are carried to objects with degree less then $n$. Altogether we see that the image of $F^N$ bifurcates: either $\ell$ carries all morphisms of $\lambda$ to isomorphisms, in which case $F^{N}(\lambda, \ell) = (\lambda, \ell)$ and $F^{N-1}(\lambda, \ell) \in   \Un \left( (\para^{\sf op})_{= N} \to {\sf Quiv}^{\sf op} \xrightarrow{\hom_\Cat(\_, \cR_{=n})} \spaces \right)  $, or $F^{N}(\lambda, \ell) \in  \Un_{\le n-1, \le N}$. We conclude that 
        \[|\Un_{\le n-1, \le N}| \simeq \left| \Un \left( (\para^{\sf op})_{= N} \to {\sf Quiv}^{\sf op} \xrightarrow{\hom_\Cat(\_, \cR_{=n})} \spaces \right)\right| \amalg \left| \Un_{\le n-1, \le N}\right|~.\] 
        This fact and manipulations of colimits gives  
        \begin{align*}
        |\Un_{\le n}| &\simeq \left|\underset{N \to \infty} \colim \Un_{\le n, \le N}\right| \\
        &\simeq \underset{N \to \infty} \colim \left| \Un \left( (\para^{\sf op})_{= N} \to {\sf Quiv}^{\sf op} \xrightarrow{\hom_\Cat(\_, \sB \cR_{=n})} \spaces \right)\right| \amalg \left|\Un_{\le n-1, \le N} \right| 
         \stepcounter{equation}\tag{\theequation} \label{induction}  \\
        &\simeq \left|\underset{N \to \infty} \colim \left[ \Un \left( (\para^{\sf op})_{= N} \to {\sf Quiv}^{\sf op} \xrightarrow{\hom_\Cat(\_, \sB \cR_{=n})} \spaces \right) \right]\right| \amalg \left|\Un_{\le n-1}\right|
        \stepcounter{equation}\tag{\theequation} \label{Un<N} \\
        &\simeq  \left| \Un \left( \underset{N \to \infty} \colim \left[(\para^{\sf op})_{= N}\right] \to {\sf Quiv}^{\sf op} \xrightarrow{\hom_\Cat(\_, \sB \cR_{=n})} \spaces \right) \right| \amalg |\Un_{\le n-1}| 
        \stepcounter{equation}\tag{\theequation} \label{exponentiable} \\
        &\simeq \left| \Un \left( \para^{\sf op} \to {\sf Quiv}^{\sf op} \xrightarrow{\hom_\Cat(\_, \sB \cR_{=n})} \spaces \right)\right| \amalg \left|\Un_{\le n-1}\right| \\
        :&= |\Un_{=n}| \amalg |\Un_{\le n-1}|
        \stepcounter{equation}\tag{\theequation} \label{fullsplit}  ~.
        \end{align*} 
    Here \cref{induction} is due to the above discussion; \cref{Un<N} is exploiting the commutativity of the geometric realization and disjoint union and applying \cref{colimNUn}; \cref{exponentiable} is again because cartesian fibrations are exponentiable. This equivalence along with our inductive hypothesis gives
    \begin{align*}
        {\sf HH}(\cR_{\le n}) &\simeq \int_{S^1} \cR_{\le n} \\
        &\simeq |\Un_{\le n}| \\
        &\simeq \left|  \Un_{= n} \right| \amalg \left|\Un_{\le n-1}\right|
        \stepcounter{equation}\tag{\theequation} \label{expansion}\\ 
        &\simeq \left| \Un_{=n} \right| \amalg {\sf HH}(\cR_{\le n-1}) 
        \stepcounter{equation}\tag{\theequation} \label{Un=HH}\\
        &\simeq \left| \Un_{=n} \right| \amalg \left[\coprod_{0 \le q\le n-1}\coprod_{[r]\in\cR_{= (q-1)}} \coprod_{[\phi]\in \Aut_{\cR}(r)_{/^{\sf conj}}} \sB \sZ_{\Aut_{\cR}(r)}(\phi) \right]   .
        \stepcounter{equation}\tag{\theequation} \label{inductivestep}
    \end{align*}
    Here \cref{expansion} is due to \cref{fullsplit}; \cref{Un=HH} uses \cref{Un<n}; \cref{inductivestep} is utilizing the inductive hypothesis.
    To finally finish our inductive argument, we need to show $|\Un_{=n}|$ is isomorphic with 
    \[\coprod_{[r]\in\cR_{= n}} \coprod_{[\phi]\in \Aut_{\cR}(r)_{/^{\sf conj}}} \sB \sZ_{\Aut_{\cR}(r)}(\phi)~.\] 
    
    Towards this, we note $\left| \Un_{=n} \right| \simeq \int_{S^1} \sB \cR_{=n} \simeq \int_{S^1}^\alpha \cR_{=n}$ as shown in \cite{ayala2025symmetriescyclicnerve}. Then by non-abelian Poincar\'e Duality $\int_{S^1}^\alpha \cR_{=n} \simeq \map(S^1, \sB \cR_{=n})$. Since $\cR_{=n}$ is a groupoid, the circle is connected, and \cref{MapS1}:
\begin{align*}
    \map(S^1, \sB \cR_{=n}) &\simeq \map(S^1, \coprod_{[r]\in \cR_{=n}} \sB\Aut_{\cR}(r)) \\
    &\simeq \coprod_{[r]\in \cR_{=n}}\map(S^1, \sB\Aut_{\cR}(r)) \\&\overset{(\ref{MapS1})}\simeq \coprod_{[r]\in\cR_{= n}} \coprod_{[\phi]\in \Aut_{\cR}(r)_{/^{\sf conj}}} \sB \sZ_{\Aut_{\cR}(r)}(\phi) ~.
\end{align*}
  All of these facts together give an equivalence: 
    \[
    \left| \Un_{=n} \right| \simeq \coprod_{[r]\in\cR_{= n}} \coprod_{[\phi]\in \Aut_{\cR}(r)_{/^{\sf conj}}} \sB \sZ_{\Aut_{\cR}(r)}(\phi) ~.
    \]
    An application of \cref{inductivestep} leads to the desired result to complete the induction:
    \begin{align*}
        {\sf HH}(\cR_{\le n}) &\simeq  \underset{0 \geq r \geq n }  \coprod \hspace{1ex} \coprod_{[r]\in\cR_{= n}} \coprod_{[\phi]\in \Aut_{\cR}(r)_{/^{\sf conj}}} \sB \sZ_{\Aut_{\cR}(r)}(\phi) ~.
    \end{align*}
\end{proof}

\end{document}